\documentclass[twoside]{amsart}

\usepackage[T1]{fontenc}

\usepackage{amssymb,colonequals,bm}
\usepackage{enumerate}
\usepackage{hyperref}

\usepackage[capitalise]{cleveref}

\crefformat{equation}{(#2#1#3)}
\crefrangeformat{equation}{(#3#1#4--#5#2#6)}
\crefmultiformat{equation}{(#2#1#3}{ \& #2#1#3)}{, #2#1#3}{ \& #2#1#3)}

\crefformat{enumi}{(#2#1#3)}
\crefrangeformat{enumi}{(#3#1#4--#5#2#6)}
\crefmultiformat{enumi}{(#2#1#3}{ \& #2#1#3)}{, #2#1#3}{ \& #2#1#3)}

\crefname{construction}{Construction}{Constructions}
\crefname{discussion}{}{}

\newtheorem*{rep@theorem}{\rep@title}
\newcommand{\newreptheorem}[2]{%
\newenvironment{rep#1}[1]{%
 \def\rep@title{#2 \ref{##1}}%
 \begin{rep@theorem}}%
 {\end{rep@theorem}}}

\newtheorem{theorem}[subsection]{Theorem}
\newtheorem{proposition}[subsection]{Proposition}
\newtheorem{lemma}[subsection]{Lemma}
\newtheorem{corollary}[subsection]{Corollary}

\newreptheorem{theorem}{Theorem}

\theoremstyle{definition}
\newtheorem{definition}[subsection]{Definition}

\newtheorem{discussion}[subsection]{}

\theoremstyle{remark}

\newtheorem{facts}[subsection]{Facts}

\numberwithin{equation}{subsection}

\usepackage{tikz-cd}

\tikzcdset{
  column sep/wide/.initial=+6em,
  column sep/Wide/.initial=+7.2em,
  column sep/wider/.initial=+8.4em,
  column sep/Wider/.initial=+9.6em,
  column sep/broad/.initial=+10.8em,
  column sep/Broad/.initial=+12em,
  column sep/broader/.initial=+13.2em,
  column sep/Broader/.initial=+14.4em,
  row sep/wide/.initial=+4.5em,
  row sep/Wide/.initial=+5.4em,
  row sep/wider/.initial=+6.3em,
  row sep/Wider/.initial=+7.2em,
}

\tikzcdset{ampersand replacement=\&,nodes in empty cells}

\renewcommand{\epsilon}{\varepsilon}
\renewcommand{\phi}{\varphi}
\renewcommand{\theta}{\vartheta}

\renewcommand{\mod}[1]{\operatorname{mod}(#1)}
\newcommand{\Mod}[1]{\operatorname{Mod}(#1)}

\newcommand{\grmod}[1]{\operatorname{grmod}(#1)}
\newcommand{\grMod}[1]{\operatorname{grMod}(#1)}

\newcommand{\dcat}[2][]{\operatorname{D}_{#1}(#2)}
\newcommand{\dbcat}[1]{\dcat[b]{#1}}

\newcommand{\cat}[1]{{\mathsf{#1}}}
\newcommand{\op}[1]{#1^{op}}
\newcommand{\set}[2]{\left\{#1 \,\middle|\, #2\right\}}

\newcommand{\BZ}{\mathbb{Z}}

\newcommand{\bmchi}{\bm{\chi}}

\newcommand{\bmf}{\bm{f}}

\DeclareMathOperator{\coh}{H}
\DeclareMathOperator{\colim}{colim}

\DeclareMathOperator{\ctr}{Z}
\DeclareMathOperator{\End}{End}
\newcommand{\Ext}[4][*]{\operatorname{Ext}^{#1}_{#2}(#3,#4)}
\newcommand{\hh}[2][*]{\operatorname{HH}^{#1}(#2)}
\newcommand{\Hom}[4][]{\operatorname{Hom}^{#1}_{#2}(#3,#4)}
\DeclareMathOperator{\id}{id}
\newcommand{\lotimes}{\otimes^{\mathbf{L}}}
\DeclareMathOperator{\Nat}{Nat}

\newcommand{\susp}{\Sigma}
\DeclareMathOperator{\thick}{thick}

\makeatletter
\@namedef{subjclassname@2020}{%
 \textup{2020} Mathematics Subject Classification}
\makeatother

\title[Graded Brown representability]{Brown representability for triangulated categories with a linear action by a graded ring}
\date{\today}

\author[J.~C.~Letz]{Janina C.~Letz}
\address{Janina C. Letz\\
Fakult\"at f\"ur Mathematik\\ 
Universit\"at Bielefeld\\ 
33501 Bielefeld\\ 
Germany.}
\email{jletz@math.uni-bielefeld.de}


\thanks{Partly supported by the NSF grant DMS-1700985, and the Alexander von Humboldt Foundation in the framework of an Alexander von Humboldt Professorship endowed by the German Federal Ministry of Education and Research.}

\begin{document}

\begin{abstract}
In this paper we give necessary and sufficient conditions for a functor to be representable in a strongly generated triangulated category which has a linear action by a graded ring, and we discuss some applications and examples. 
\end{abstract}

\maketitle

\section{Introduction}

Every object $X$ in a category induces a contravariant functor into the category of sets that sends an object $Y$ to the set $\Hom{}{Y}{X}$. Any functor that is naturally isomorphic to such a functor is called \emph{representable}. There are a number of results in various settings, called Brown representability, when every `reasonable' functor is representable. The first such result is due to Brown; see \cite{Brown:1962}.

The first Brown representability result for triangulated categories was established by Neeman \cite[Theorem 3.1]{Neeman:1996}. The work on hand was motivated by \cite[Theorem~1.3]{Bondal/VanDenBergh:2003} and \cite[4.3]{Rouquier:2008}.

\begin{reptheorem}{grddRep}
Let $R$ be a $\BZ$-graded graded-commutative noetherian ring and $\cat{T}$ a graded $R$-linear triangulated category, that is strongly generated, Ext-finite and idempotent complete. Then a graded $R$-linear cohomological functor $\mathsf{f} \colon \op{\cat{T}} \to \grMod{R}$ is graded representable if and only if $\mathsf{f}$ only takes values in $\grmod{R}$.
\end{reptheorem}

In contrast to the previous works, we characterize the \emph{graded representable} functors; those are the functors naturally isomorphic to
\begin{equation*}
\coprod_{d \in \BZ} \Hom{\cat{T}}{-}{\susp^d X}
\end{equation*}
for some object $X$. The result is proved in \cref{sec:Representability}. Without the assumption that $R$ is noetherian and $\cat{T}$ Ext-finite, we obtain necessary, through not sufficient, conditions for a functor to be graded representable; see \cref{lfpIdemCompGrddRpn}. 

The study of representability is motivated by the fact, that the characterization of representable functors in a triangulated category $\cat{T}$ yields the existence of a right adjoint functor to a functor $\cat{S} \to \cat{T}$. For a nice discussion on this see \cite[Introduction]{Neeman:2001}. In \cref{sec:Applications} we show the same holds in the graded setting. 

Finally we discuss some examples where \cref{grddRep} yields new insight: When $G$ is a finite group and $R$ a commutative noetherian ring, then $\dbcat{\mod{RG}}$ is Ext-finite as a $\coh^*(G,R)$-linear category. In the second example we consider the action of Hochschild cohomology $\hh[*]{R/Q}$ on $\dbcat{\mod{R}}$, when $Q$ is a regular ring and $R = Q/(\bmf)$ a quotient by a regular sequence. 

\section{Representable Functors in the Graded Setting} \label{sec:Representability}

Let $\cat{T}$ be a triangulated category with suspension functor $\susp$. 

\begin{discussion}
For objects $X$ and $Y$ in $\cat{T}$ we write
\begin{equation} \label{DefExtTriangulated}
\Ext{\cat{T}}{X}{Y} \colonequals \coprod_{d \in \BZ} \Hom{\cat{T}}{X}{\susp^d Y}\,.
\end{equation}
When $\cat{T} = \dcat{R}$, the derived category of modules over a ring $R$, and $X$ and $Y$ are $R$-modules viewed as objects in $\dcat{R}$ via the natural embedding, then this coincides with the classical Ext-groups.
\end{discussion}

\begin{discussion}
Let $R$ be a $\BZ$-graded graded-commutative ring. This means $R$ decomposes as
\begin{equation*}
R = \coprod_{d \in \BZ} R_d\,,
\end{equation*}
and the Koszul sign rule holds
\begin{equation*}
rs = (-1)^{de} sr \quad\text{for } r \in R_d \text{ and }s \in R_e\,.
\end{equation*}
We say $r \in R_d$ is an homogeneous element of degree $d$. 
\end{discussion}

\begin{discussion}
A triangulated category $\cat{T}$ is \emph{graded $R$-linear} if
\begin{enumerate}
\item for any objects $X$ and $Y$ in $\cat{T}$, the abelian group $\Ext{\cat{T}}{X}{Y}$ is a graded $R$-module with the grading given by the coproduct in \cref{DefExtTriangulated}, and
\item composition is $R$-bilinear. 
\end{enumerate}

This data is equivalent to a ring homomorphism $R \to \ctr(\cat{T})$, where
\begin{equation*}
\ctr(\cat{T}) \colonequals \coprod_{d \in \BZ} \set{\eta \colon \id_\cat{T} \to \susp^d}{\eta \susp = (-1)^d \susp \eta}
\end{equation*}
is the \emph{graded center of $\cat{T}$}. More precisely, a ring homomorphism $\varphi \colon R \to \ctr(\cat{T})$ yields an $R$-action on $\Ext{\cat{T}}{X}{Y}$ via
\begin{gather*}
r \cdot - \colon \Ext{\cat{T}}{X}{Y} \to \Ext{\cat{T}}{X}{\susp^d Y} = \Ext{\cat{T}}{X}{Y}[d] \,, \\
f \mapsto (\susp^d f) \circ \varphi(r)_X = \varphi(r)_Y \circ f
\end{gather*}
for any homogeneous element $r \in R$. Conversely, any homogeneous element $r \in R$ yields a natural transformation $\eta$ given by
\begin{equation*}
\eta_X \colonequals r \cdot \id_X \colon X \to \susp^{|r|} X
\end{equation*}
for any $X \in \cat{T}$. It is straightforward to check, that these identifications are well-defined and mutually inverse. The graded center has been studied in a number of works; for example \cite{Benson/Iyengar/Krause:2008,Buchweitz/Flenner:2008}. 
\end{discussion}

\begin{discussion}
We denote by $\grMod{R}$ the category of graded $R$-modules, and by $\grmod{R}$ its full subcategory of finitely generated $R$-modules. The $n$th shift $M[n]$ of a graded $R$-module $M$ is given by $(M[n])_d = M_{n+d}$.

The suspension functor of a graded $R$-linear category $\cat{T}$ in the first component of $\Ext{\cat{T}}{-}{-}$ corresponds to the negative shift in $\grMod{R}$: 
\begin{equation*}
\Ext{\cat{T}}{\susp^n X}{Y} \cong \Ext{\cat{T}}{X}{Y}[-n]\,.
\end{equation*}
\end{discussion}

\begin{discussion}
A functor $\mathsf{f} \colon \op{\cat{T}} \to \grMod{R}$ is \emph{graded $R$-linear}, if
\begin{enumerate}
\item the induced map $\Ext{\cat{T}}{X}{Y} \to \Ext{R}{\mathsf{f}(Y)}{\mathsf{f}(X)}$ is a map of graded $R$-modules, and
\item the suspension becomes the negative shift under $\mathsf{f}$, that is
\begin{equation*}
\mathsf{f}(\susp^n X) = \mathsf{f}(X)[-n]\,.
\end{equation*}
\end{enumerate}
The functor $\mathsf{f}$ is \emph{cohomological} if $\mathsf{f}$ applied to any exact triangle yields a long exact sequence of graded $R$-modules. 

Without explicitly stating we always assume that a natural transformation between graded $R$-linear functors respects this structure. 
\end{discussion}

\begin{definition}
A functor $\op{\cat{T}} \to \grMod{R}$ is \emph{graded representable} if it is naturally isomorphic to
\begin{equation*}
\mathsf{g}_X \colonequals \Ext{\cat{T}}{-}{X} \colon \op{\cat{T}} \to \grMod{R}
\end{equation*}
for some object $X$ in $\cat{T}$.
\end{definition}

When $\cat{T}$ is graded $R$-linear, then any graded representable functor is graded $R$-linear. 

A graded $R$-linear functor $\mathsf{f} \colon \op{\cat{T}} \to \grMod{R}$ is graded representable if and only if $\mathsf{f}_d \colon \op{\cat{T}} \to \Mod{R_0}$ is representable for an(y) arbitrary integer $d$. The functors $\mathsf{f}_d$ are the degree $d$ part of $\mathsf{f}$, that is $\mathsf{f}_d(X) \colonequals \mathsf{f}(X)_d$. Since $\mathsf{f}$ is graded $R$-linear, the degree $d$ part $\mathsf{f}_d$ for an integer $d$ encodes all the information of $\mathsf{f}$, that is
\begin{equation*}
\mathsf{f}_d(\susp^e X) = \mathsf{f}(\susp^e X)_d = \mathsf{f}(X)[-e]_d = \mathsf{f}(X)_{d-e} = \mathsf{f}_{d-e}(X)\,.
\end{equation*}

\begin{theorem} \label{grddRep}
Let $R$ be a $\BZ$-graded graded-commutative noetherian ring and $\cat{T}$ a graded $R$-linear triangulated category, that is strongly generated, Ext-finite and idempotent complete. Then a graded $R$-linear cohomological functor $\mathsf{f} \colon \op{\cat{T}} \to \grMod{R}$ is graded representable if and only if $\mathsf{f}$ is locally finite.
\end{theorem}

Before we give a proof, we recall some definitions and properties:

\begin{discussion}
A $\BZ$-graded ring $R$ is noetherian, if and only if $R_0$ is noetherian and $R$ is finitely generated as an $R_0$-algebra; see for example \cite[Corollaire~(2.1.5)]{Grothendieck:1961} or \cite[Theorem~1.5.5]{Bruns/Herzog:1998}. In particular, such a ring is bounded below. 
\end{discussion}

\begin{discussion}
A graded $R$-linear triangulated category $\cat{T}$ is \emph{Ext-finite} if for all $X, Y \in \cat{T}$, the graded $R$-module $\Ext{\cat{T}}{X}{Y}$ is finitely generated. 

A triangulated category $\cat{T}$ is \emph{idempotent complete} if for every object $X$ in $\cat{T}$ and every idempotent $e \in \End_\cat{T}(X)$, that is $e^2 = e$, there exists an object $Y$ and maps
\begin{equation*}
i \colon Y \to X \quad\text{and}\quad p \colon X \to Y
\end{equation*}
such that $p \circ i = \id_Y$ and $i \circ p = e$. 
\end{discussion}

\begin{discussion}
A subcategory $\cat{S} \subseteq \cat{T}$ is \emph{thick}, if it is triangulated and closed under retracts. Since the intersection of thick subcategories is thick, there exists a smallest thick subcategory of $\cat{T}$ containing an object $G$, which we denote by $\thick(G)$. We say $G$ \emph{finitely builds} an object $X$ in $\cat{T}$ when $X \in \thick(G)$. 

There is an exhaustive filtration of $\thick_R(G)$: Let $\thick^1(G)$ be the smallest full subcategory containing $G$ that is closed under finite coproducts, retracts and suspension. Then
\begin{equation*}
\thick^n(G) \colonequals \set{X \in \cat{T}}{\begin{gathered}
\text{there exists } X' \in \cat{T} \text{ and an exact triangle } \\
Y \to X \oplus X' \to Z \to \susp Y \\
\text{ such that } Y \in \thick^{n-1}(G) \text{ and } Z \in \thick^1(G)
\end{gathered}}\,.
\end{equation*}
These are full subcategories and form an exhaustive filtration of $\thick(G)$; cf.\@ \cite{Bondal/VanDenBergh:2003,Avramov/Buchweitz/Iyengar/Miller:2010}. In particular, if $X$ lies in $\thick(G)$, then there exists an integer $n$, such that $X \in \thick^n(G)$. 

A triangulated category $\cat{T}$ is \emph{strongly generated}, if there exists an object $G$ in $\cat{T}$ and a non-negative integer $n$, such that $\cat{T} = \thick^n(G)$. The object $X$ is a \emph{strong generator} of $\cat{T}$; cf.\@ \cite{Rouquier:2008}.
\end{discussion}

In the reminder of this section, we give a proof of \cref{grddRep}. We fix a $\BZ$-graded graded-commutative ring $R$, a graded $R$-linear triangulated category $\cat{T}$, and a graded $R$-linear cohomological functor $\mathsf{f} \colon \op{\cat{T}} \to \grMod{R}$. 

\begin{lemma}[Graded version of Yoneda's lemma] \label{GradedYonedasLemma}
For any $X \in \cat{T}$, the map
\begin{equation*}
\Nat(\mathsf{g}_X,\mathsf{f}) \to \mathsf{f}_0(X) \quad\text{given by}\quad \eta \mapsto \eta(X)(\id_X)
\end{equation*}
is an isomorphism of abelian groups.
\end{lemma}
\begin{proof}
For $u \in \mathsf{f}_0(X)$, we define a natural transformation
\begin{equation*}
\eta_u \colon \mathsf{g}_X \to \mathsf{f} \quad\text{as}\quad \eta_u(Y)(f) \colonequals \mathsf{f}(f)(u)
\end{equation*}
where $Y \in \cat{T}$ and $f \in \Ext{\cat{T}}{Y}{X}$. Since $u$ is a degree zero element, the map $\eta_u(Y)$ is homogeneous. It is straightforward to verify that this is the inverse of the map in the claim and both are maps of abelian groups.
\end{proof}

In particular, any morphism $f \colon X \to Y$ corresponds to a natural transformation
\begin{equation*}
f_* \colon \mathsf{g}_X \to \mathsf{g}_Y
\end{equation*}
given by post-composition. 

\begin{discussion}
Adapting the definitions in \cite[Section~4]{Rouquier:2008}, a graded $R$-linear functor $\mathsf{f} \colon \op{\cat{T}} \to \grMod{R}$ is
\begin{itemize}
\item \emph{locally finitely generated}, if for every $X$ in $\cat{T}$ there exists $Y$ in $\cat{T}$ and a natural transformation $\zeta \colon \mathsf{g}_Y \to \mathsf{f}$ such that $\zeta(X)$ is surjective, 
\item \emph{locally finitely presented}, if it is locally finitely generated and the kernel of any natural transformation $\mathsf{g}_Y \to \mathsf{f}$ is locally finitely generated, and
\item \emph{locally finite}, if $\mathsf{f}$ only takes values in $\grmod{R}$. 
\end{itemize}
\end{discussion}

When $\mathsf{f} \colon \op{\cat{T}} \to \grMod{R}$ is locally finitely generated or locally finitely presented, then $\mathsf{f}_d \colon \op{\cat{T}} \to \Mod{R_0}$ is locally finitely generated or locally finitely presented in the sense of \cite[Section~4]{Rouquier:2008}, respectively. The same need not hold for locally finite, for examples see \cref{sec:Applications}. 

If $\cat{T}$ is Ext-finite, then any graded representable functor is locally finite. Without the assumption that $\cat{T}$ is Ext-finite, we can make the following statement:

\begin{lemma} \label{grddRpnLocfp}
Any graded representable functor is locally finitely presented.
\end{lemma}
\begin{proof}
It is clear that a graded representable functor is locally finitely generated. Let $\mathsf{g}_X$ be a graded representable functor, and $\mathsf{g}_Y \to \mathsf{g}_X$ a natural transformation. By Yoneda's \cref{GradedYonedasLemma}, this corresponds to a morphism $Y \to X$. If we complete this to an exact triangle $Z \to Y \to X \to \susp Z$, the sequence
\begin{equation*}
\mathsf{g}_Z \to \mathsf{g}_Y \to \mathsf{g}_X
\end{equation*}
is exact on $\cat{T}$. In particular, the kernel of $\mathsf{g}_Y \to \mathsf{g}_X$ is locally finitely generated.
\end{proof}

\begin{lemma} \label{ExistenceSurjNatTrans}
If $\mathsf{f}$ is locally finite, then $\mathsf{f}$ is locally finitely generated. 
\end{lemma}
\begin{proof}
Let $X$ be an object in $\cat{T}$. Then the $R$-module $\mathsf{f}(X)$ is finitely generated, and we can choose a finite set of homogeneous generators $x_1, \dots, x_n$ of $\mathsf{f}(X)$ in degrees $d_1, \dots, d_n$. Set
\begin{equation*}
Y \colonequals \coprod_{j=1}^n \susp^{d_j} X\,.
\end{equation*}
For every generator $x_j$, we obtain canonical maps
\begin{equation*}
\susp^{d_j} X \xrightarrow{i_j} Y \xrightarrow{p_j} \susp^{d_j} X
\end{equation*}
whose composition is the identity map on $\susp^{d_j} X$. Let $y \in \mathsf{f}(Y)$ be the canonical element, for which
\begin{equation*}
x_j = \mathsf{f}(i_j)(y) \quad\text{for } 1 \leq j \leq n\,.
\end{equation*}
Because of the suspensions introduced in the definition of $Y$, the element $y$ is homogeneous of degree 0. By Yoneda's lemma~\ref{GradedYonedasLemma}, the element $y$ corresponds to the natural transformation $\zeta \colon \mathsf{g}_Y \to \mathsf{f}$ with $\zeta(Y)(\id_Y) = y$. Then $\zeta(y)(i_j) = x_j$, and $\zeta(X)$ is surjective. That is $\mathsf{f}$ is locally finitely generated. 
\end{proof}

In general a locally finite functor need not be locally finitely presented. This requires further assumptions on $R$ and $\cat{T}$: 

\begin{lemma} \label{LocallyFiniteThenLFP}
If $R$ is noetherian and $\cat{T}$ Ext-finite, then a locally finite functor $\mathsf{f} \colon \op{\cat{T}} \to \grMod{R}$ is locally finitely presented.
\end{lemma}
\begin{proof}
By \cref{ExistenceSurjNatTrans} the functor $\mathsf{f}$ is locally finitely generated. Let $\mathsf{g}_Y \to \mathsf{f}$ be a natural transformation. We set
\begin{equation*}
\mathsf{f}'(X) \colonequals \ker(\mathsf{g}_Y(X) \to \mathsf{f}(X))\,.
\end{equation*}
Since $\cat{T}$ is Ext-finite, the $R$-module $\mathsf{g}_Y(X)$ is finitely generated. By assumption on $\mathsf{f}$, so is $\mathsf{f}(X)$. Since $R$ is noetherian, the kernel $\mathsf{f}'(X)$ is also finitely generated. Thus $\mathsf{f}'$ is a locally finite functor and by \cref{ExistenceSurjNatTrans} it is locally finitely generated. In particular, $\mathsf{f}$ is locally finitely presented. 
\end{proof}

\begin{discussion}
Let $(\mathsf{f}_i,\eta_i)_{i > 0}$ be a direct system of cohomological functors $\mathsf{f}_i \colon \cat{T} \to \cat{A}$ where $\cat{A}$ an abelian category and natural transformations $\eta_i \colon \mathsf{f}_i \to \mathsf{f}_{i+1}$. Following \cite[4.2.2]{Rouquier:2008}, a direct system $(\mathsf{f}_i,\eta_i)_{i > 0}$ is \emph{almost constant on} a subcategory $\cat{S}$ of $\cat{T}$, if for every $X \in \cat{S}$ the sequence
\begin{equation*}
0 \to \ker(\eta_i(X)) \to \mathsf{f}_i(X) \to \colim_j \mathsf{f}_j(X) \to 0
\end{equation*}
is exact for all positive integers $i$. 

A direct system $(X_i,f_i)_{i>0}$ of objects $X_i$ and morphisms $f_i \colon X_i \to X_{i+1}$ in $\cat{T}$ is \emph{almost constant on $\cat{S}$}, if the induced direct system of functors $(\mathsf{g}_{X_i},(f_i)_*)_{i>0}$ is almost constant on $\cat{S}$. 
\end{discussion}

For almost constant direct systems the following hold; see \cite[Proposition~4.13]{Rouquier:2008}.

\begin{facts} \label{almostConstant}
Let $\cat{S} \subseteq \cat{T}$ be a subcategory closed under suspension, and $(\mathsf{f}_i,\eta_i)_{i>0}$ a direct system that is almost constant on $\cat{S}$. Then
\begin{enumerate}
\item $(\mathsf{f}_{ni+r})_{i \geqslant 0}$ is almost constant on $\thick^n(\cat{S})$ for any $r > 0$, and
\item $\mathsf{f}_{n+1} \to \colim \mathsf{f}_i$ is split surjective on $\thick^n(\cat{S})$. 
\end{enumerate}
If the functors $\mathsf{f}_i$ are graded $R$-linear, the assumption that $\cat{S}$ is closed under suspension is redundant. 
\end{facts}

\begin{proposition} \label{lfpiffRetractGrddRpn}
Let $\cat{T}$ be a strongly generated, graded $R$-linear triangulated category and $\mathsf{f} \colon \op{\cat{T}} \to \grMod{R}$ a cohomological graded $R$-linear functor. Then $\mathsf{f}$ is locally finitely presented if and only if $\mathsf{f}$ is a retract of a graded representable functor. 
\end{proposition}
\begin{proof}
We assume $\mathsf{f}$ is locally finitely presented. Let $G \in \cat{T}$ be a strong generator of $\cat{T}$ with $\thick^d(G) = \cat{T}$. Then there exist $A_1 \in \cat{T}$ and a natural transformation $\zeta_1 \colon \mathsf{g}_{A_1} \to \mathsf{f}$ such that $\zeta_1(G)$ is surjective. Inductively we construct a direct system
\begin{equation*}
\mathsf{g}_{A_1} \to \mathsf{g}_{A_2} \to \cdots
\end{equation*}
with compatible natural transformations $\zeta_i \colon \mathsf{g}_{A_i} \to \mathsf{f}$: Assume we have constructed $A_i$ and $\zeta_i$ for $i \leq n$. Since $\mathsf{f}$ is locally finitely presented, there exists
\begin{equation*}
\mathsf{g}_B \to \ker(\mathsf{g}_{A_n} \to \mathsf{f})
\end{equation*}
that is surjective on $G$. This induces a natural transformation $\mathsf{g}_B \to \mathsf{g}_{A_n}$, which by the graded version of Yoneda's Lemma~\ref{GradedYonedasLemma} corresponds to a morphism $f \colon B \to A_n$. We complete this morphism to an exact triangle
\begin{equation*}
B \to A_n \to A_{n+1} \to \susp B
\end{equation*}
and apply $\mathsf{f}_0$, the degree 0 part of $\mathsf{f}$. By the graded version of Yoneda's Lemma~\ref{GradedYonedasLemma} we obtain the exact sequence
\begin{equation*}
\Nat(\mathsf{g}_B,\mathsf{f}) \leftarrow \Nat(\mathsf{g}_{A_n},\mathsf{f}) \leftarrow \Nat(\mathsf{g}_{A_{n+1}},\mathsf{f})\,.
\end{equation*}
Thus by construction of $B$ there exists a natural transformation $\zeta_{n+1}$ whose image is $\zeta_n$. 

By this construction we have
\begin{equation*}
\ker(\mathsf{g}_{A_n}(G) \to \mathsf{f}(G)) = \ker(\mathsf{g}_{A_n}(G) \to \mathsf{g}_{A_{n+1}}(G))\,.
\end{equation*}
Using this and that $\zeta_1(G)$ is surjective, it is straightforward to verify that the direct system is almost constant on $G$. Then the induced natural transformation $\colim_i \mathsf{g}_{A_i} \to \mathsf{f}$ is a natural isomorphism. By \cref{almostConstant}, the natural transformation
\begin{equation*}
\mathsf{g}_{A_{d+1}} \to \colim_i \mathsf{g}_{A_i} \xrightarrow{\sim} \mathsf{f}
\end{equation*}
is split surjective, and thus $\mathsf{f}$ is a retract of $\mathsf{g}_{A_{d+1}}$ on $\cat{T}$.

For the converse direction, we assume $\mathsf{f}$ is the retract of $\mathsf{g}_X$ for some object $X$. Then we have a canonical projection and a canonical injection
\begin{equation*}
\mathsf{g}_X \to \mathsf{f} \quad\text{and}\quad \mathsf{f} \to \mathsf{g}_X\,,
\end{equation*}
respectively. The canonical projection is surjective on $\cat{T}$, the canonical injection is injective. In particular, the canonical projection yields that $\mathsf{f}$ is locally finitely generated. Given a natural transformation $\mathsf{g}_Y \to \mathsf{f}$, its kernel coincides with the kernel of the composition $\mathsf{g}_Y \to \mathsf{f} \to \mathsf{g}_X$. By \cref{grddRpnLocfp}, any representable functor is locally finitely presented, and thus is $\mathsf{f}$. 
\end{proof}

\begin{corollary} \label{lfpIdemCompGrddRpn}
If $\cat{T}$ is additionally idempotent complete, then every locally finitely presented functor is graded representable. 
\end{corollary}
\begin{proof}
Let $\mathsf{f}$ be a locally finitely presented functor. By \cref{lfpiffRetractGrddRpn}, it is a retract of a graded representable functor $\mathsf{g}_X$. Then the natural transformation
\begin{equation*}
\mathsf{g}_X \to \mathsf{f} \to \mathsf{g}_X
\end{equation*}
corresponds to an idempotent $e \colon X \to X$. Since $\cat{T}$ is idempotent complete, there exists a retract of $Y$ of $X$, such that $e$ decomposes as the natural inclusion and projection morphism. Then $\mathsf{f} \to \mathsf{g}_X \to \mathsf{g}_Y$ is a natural isomorphism, and $\mathsf{f}$ is graded representable. 
\end{proof}

\begin{proof}[Proof of \cref{grddRep}]
Since $\cat{T}$ is Ext-finite, any graded representable functor is locally finite. For the converse, we assume $\mathsf{f}$ is locally finite. Since $R$ is noetherian and $\cat{T}$ Ext-finite, we can apply \cref{LocallyFiniteThenLFP} to obtain that $\mathsf{f}$ is locally finitely presented. Then $\mathsf{f}$ is graded representable by \cref{lfpIdemCompGrddRpn}. 
\end{proof}

\section{Applications} \label{sec:Applications}

\subsubsection*{Adjoint Functors}

As explained in \cite[Introduction]{Neeman:2001} there is a connection between representable functors and adjoint functors. In our context we obtain the following:

Let $R$ be a $\BZ$-graded graded-commutative ring. A functor $\mathsf{f} \colon \cat{S} \to \cat{T}$ between $R$-linear graded triangulated categories is \emph{graded $R$-linear}, if it is exact and the induced map
\begin{equation*}
\Ext{\cat{S}}{X}{Y} \to \Ext{\cat{T}}{\mathsf{f}(X)}{\mathsf{f}(Y)}
\end{equation*}
is a map of graded $R$-modules. 

\begin{lemma}
Let $R$ be a $\BZ$-graded graded-commutative ring, and $\cat{S}$, $\cat{T}$ graded $R$-linear triangulated categories. Suppose $\cat{T}$ is Ext-finite and every cohomological graded $R$-linear functor $\op{\cat{S}} \to \grMod{R}$, that is locally finite, is graded representable. Then every graded $R$-linear functor $\mathsf{f} \colon \cat{S} \to \cat{T}$ has a right adjoint.
\end{lemma}
\begin{proof}
We adapt the proof of \cite[Theorem~8.4.4]{Neeman:2001}. Given $Y \in \cat{T}$ we define a functor $\mathsf{h} \colon \cat{S} \to \grMod{R}$ by
\begin{equation*}
\mathsf{h}(-) \colonequals \Ext{\cat{T}}{\mathsf{f}(-)}{Y}\,.
\end{equation*}
This is a graded $R$-linear functor. Since $\cat{T}$ is Ext-finite, this functor is locally finite. So by assumption $\cat{h}$ is graded representable, that is there exists an object $\mathsf{f}'(Y) \in \cat{S}$, such that
\begin{equation*}
\Ext{\cat{T}}{\mathsf{f}(-)}{Y} \cong \Ext{\cat{S}}{-}{\mathsf{f}'(Y)}\,.
\end{equation*}
It remains to verify that $\mathsf{f}'$ is a functor and this isomorphism is natural in both components. Let $f \colon Y \to Z$ be a morphism in $\cat{T}$. Then the induced map
\begin{equation*}
\Ext{\cat{S}}{-}{\mathsf{f}'(Y)} \to \Ext{\cat{S}}{-}{\mathsf{f}'(Z)}
\end{equation*}
corresponds to a morphism $\mathsf{f}'(Y) \to \mathsf{f}'(Z)$ by Yoneda's \cref{GradedYonedasLemma}. Thus $\mathsf{f}'$ is a functor. The above isomorphism is natural by construction. So $\mathsf{f}'$ is a right adjoint of $\mathsf{f}$. 
\end{proof}

\begin{corollary}
Let $R$ be a $\BZ$-graded graded-commutative noetherian ring and $\cat{S}$, $\cat{T}$ Ext-finite graded $R$-linear triangulated categories. Suppose $\cat{S}$ is strongly generated and idempotent complete. Then every graded $R$-linear functor $\mathsf{f} \colon \cat{S} \to \cat{T}$ has a right adjoint. \qed
\end{corollary}

\subsubsection*{Derived Category}

Let $R$ be a commutative noetherian ring and $A$ an $R$-algebra that is finitely generated as an $R$-module. Then $A$ is noetherian; see for example \cite[Theorem~3.7]{Matsumura:1989}. The bounded derived category of finitely generated modules over $A$, denoted by $\dbcat{\mod{A}}$, has a canonical structure as an $R$-linear category, and the $R$-module $\Hom{\dbcat{\mod{A}}}{X}{Y} = \Ext[0]{R}{X}{Y}$ is finitely generated for any $X,Y$. In general, the category $\dbcat{\mod{A}}$ need not be Ext-finite as an $R$-linear category. By \cite[Corollary~2.10]{Balmer/Schlichting:2001}, the category $\dbcat{\mod{A}}$ is idempotent complete. 

\begin{discussion}
In general the question whether $\dbcat{\mod{A}}$ is strongly generated is rather difficult. When $A$ is artinian, then $\dbcat{\mod{A}}$ is strongly generated by \cite[Proposition~7.37]{Rouquier:2008}. When $A = R$ is a commutative notherian ring, then $\dbcat{\mod{R}}$ is strongly generated when $R$ is either essentially of finite type over a field or over an equicharacteristic excellent local ring; see \cite[Main Theorem]{Aihara/Takahashi:2015} and \cite[Corollary~7.2]{Iyengar/Takahashi:2016}.
\end{discussion}

In the following we discuss two examples in which $\dbcat{\mod{A}}$ is Ext-finite for some cohomology ring connected to $A$. 

\subsubsection*{Finite Group over a Commutative Ring}

We consider $A = RG$, the group algebra of a finite group $G$.

\begin{discussion} \label{Grp:Linearity}
The group cohomology of the group algebra $RG$ with coefficients in an $RG$-complex $M$ is
\begin{equation*}
\coh^*(G,M) \colonequals \Ext[*]{RG}{R}{M}\,.
\end{equation*}
When $M = R$ this is a $\BZ$-graded graded-commutative ring, and every $\coh^*(G,M)$ is a graded $\coh^*(G,R)$-module. In particular, for $RG$-complexes $X$, $Y$ the identification
\begin{equation*}
\Ext{\dbcat{\mod{RG}}}{X}{Y} = \Ext{RG}{X}{Y} \cong \coh^*(G,\Hom{R}{X}{Y})
\end{equation*}
holds and the cohomology ring $\coh^*(G,R)$ acts on any Ext-module; see for example \cite[Proposition~3.1.8]{Benson:1998}. So the bounded derived category of finitely generated $RG$-modules $\dbcat{\mod{RG}}$ is graded $\coh^*(G,R)$-linear.
\end{discussion}

\begin{discussion} \label{Grp:FiniteGend}
By \cite{Venkov:1959,Evens:1961}, the group cohomology ring $\coh^*(G,R)$ is noetherian, and $\coh^*(G,M)$ is finitely generated over $\coh^*(G,R)$ for every finitely generated $RG$-module $M$. In particular, the derived category $\dbcat{\mod{RG}}$ is Ext-finite as a graded $\coh^*(G,R)$-linear triangulated category. 
\end{discussion}

\begin{corollary} \label{Grp:Representability}
Let $R$ be a commutative notherian ring and $G$ a finite group. If $\dbcat{\mod{RG}}$ is strongly generated, then a graded $\coh^*(G,R)$-linear functor 
\begin{equation*}
\mathsf{f} \colon \dbcat{\mod{RG}} \to \grMod{\coh^*(G,R)}
\end{equation*}
is graded representable if and only if $\mathsf{f}$ is locally finite. \qed
\end{corollary}

\subsubsection*{Regular ring modulo a regular sequence} 

We consider $R = A$ a commutative noetherian ring. 

\begin{discussion}
The category $\dbcat{\mod{R}}$ is Ext-finite over $R$, if and only if the Ext-modules $\Ext{R}{X}{Y}$ are bounded for all $X$ and $Y$ in $\dbcat{\mod{R}}$. That is precisely when $R$ is regular: When $R$ is regular the Ext-modules are bounded by definition. For the converse, for every $X$ in $\dbcat{\mod{R}}$ the Ext-module $\Ext{R}{X}{R/\mathfrak{p}}$ is bounded and $X_\mathfrak{p}$ has finite projective dimension for any prime ideal $\mathfrak{p}$ of $R$. Then $X$ has finite projective dimension; see \cite[Lemma~4.5]{Bass/Murthy:1967} for modules, and \cite[Theorem~4.1]{Avramov/Iyengar/Lipman:2010} and \cite[Theorem~3.6]{Letz:2021} for complexes. 

When $R$ is regular, the bounded derived category $\dbcat{\mod{R}}$ is strongly generated if and only if $R$ is a strong generator. The later holds precisely when $R$ has finite global dimension, that is $R$ has finite Krull dimension. Then Rouquier's representability theorem \cite[Corollary~4.18]{Rouquier:2008} applies. 
\end{discussion}

\begin{discussion}
Suppose $R = Q/(\bmf)$ is the quotient of a regular ring $Q$ by a regular sequence $\bmf = f_1, \ldots, f_c$. Then there exist cohomological operators $\bmchi = \chi_1, \ldots, \chi_c$ in degree 2, such that for $X,Y$ in $\dbcat{\mod{R}}$ the graded modules $\Ext{R}{X}{Y}$ are finitely generated over the noetherian graded ring $R[\bmchi]$; see \cite[Theorem~(4.2)]{Avramov/Gasharov/Peeva:1997}. In particular, the category $\dbcat{\mod{R}}$ is $R[\bmchi]$-linear and Ext-finite. 

The ring of cohomological operators coincides with the Hochschild cohomology
\begin{equation*}
R[\bmchi] \cong \hh{R/Q} \colonequals \Ext{R \lotimes_Q R}{R}{R}\,;
\end{equation*}
see \cite[Section~3]{Avramov/Buchweitz:2000a}.
\end{discussion}

\begin{corollary} \label{CI:Representability}
Let $R = Q/(\bmf)$ be the quotient of a regular ring $Q$ by a regular sequence $\bmf = f_1, \ldots, f_c$ with cohomological operators $\bmchi$. If $\dbcat{\mod{R}}$ is strongly generated, then any graded $R[\bmchi]$-linear functor $\mathsf{f} \colon \dbcat{\mod{R}} \to \grMod{R[\bmchi]}$ is graded representable if and only if $\mathsf{f}$ is locally finite. \qed
\end{corollary}

\begin{discussion}
For \cref{Grp:Representability,CI:Representability} it is crucial that the ring action on the derived category is graded, since the Ext-modules need not be not bounded. In particular, \cref{Grp:Representability,CI:Representability} are not consequences of \cite[4.3]{Rouquier:2008}, but require \cref{grddRep}. 
\end{discussion}

\bibliographystyle{amsplain}
\bibliography{References}

\end{document}